\documentclass{amsart}

\usepackage{algorithm}
\usepackage{algorithmicx}
\usepackage{algpseudocode}
\usepackage{array}
\usepackage[english]{babel}
\usepackage{bbm}
\usepackage{framed}
\usepackage{geometry}
\usepackage{graphicx}
\usepackage{hyperref}
\usepackage{subcaption}
\usepackage{xcolor}

\newcommand{\C}{\mathbb C}

\renewcommand{\P}{\mathbb P}
\newcommand{\R}{\mathbb R}
\newcommand{\1}{\mathbbm 1}
\newcommand{\half}{\mbox{$\frac 1 2$}}
\renewcommand{\Im}{\operatorname{Im}}
\renewcommand{\Re}{\operatorname{Re}}

\renewcommand{\Re}{\operatorname{Re}}

\newtheorem{theorem}{Theorem}[section]
\newtheorem{lemma}[theorem]{Lemma}
\newtheorem{proposition}[theorem]{Proposition}

\theoremstyle{definition}
\newtheorem{definition}[theorem]{Definition}
\theoremstyle{remark}
\newtheorem{remark}[theorem]{Remark}

\newcommand{\diag}{\operatorname{diag}}

\newcommand{\Cov}{\operatorname{Cov}}

\title{Non-reversible Metropolis-Hastings}
\author{Joris Bierkens}
\thanks{This research has received support from the  European  Union  project \# FP7-ICT-270327 (Complacs), and the EPSRC under the CRiSM grant: EP/D002060/1.}
\address{Joris Bierkens, University of Warwick, Department of Statistics, Coventry CV4 7AL United Kingdom, j.bierkens@warwick.ac.uk}

\begin{document}
\begin{abstract}
The classical Metropolis-Hastings (MH) algorithm can be extended to generate non-reversible Markov chains. This is achieved by means of a modification of the acceptance probability, using the notion of vorticity matrix. The resulting Markov chain is non-reversible. Results from the literature on asymptotic variance, large deviations theory and mixing time are mentioned, and in the case of a large deviations result, adapted, to explain how non-reversible Markov chains have favorable properties in these respects.

We provide an application of NRMH in a continuous setting by developing the necessary theory and applying, as first examples, the theory to Gaussian distributions in three and nine dimensions. The empirical autocorrelation and estimated asymptotic variance for NRMH applied to these examples show significant improvement compared to MH with identical stepsize.
\end{abstract}

\maketitle


Keywords and keyphrases: Markov Chain Monte Carlo; MCMC; Metropolis-Hastings; non-reversible Markov processes; asymptotic variance; large deviations; Langevin sampling

\section{Introduction}
The Metropolis-Hastings (MH) algorithm \cite{Metropolis1953,Hastings1970} is a Markov chain Monte-Carlo (MCMC) method of profound importance to many fields of mathematics such as Bayesian inference and statistical mechanics \cite{Diaconis1998,Diaconis2008,Levin2009}. The applicability of MH to a particular computational problem depends on the efficiency of the Markov chain that is generated by the algorithm. The chains generated by the classical Metropolis-Hastings algorithm are \emph{reversible}, or, in other words, satisfy \emph{detailed balance}; in fact, this reversibility is instrumental in showing that the resulting chains have the right invariant probability distribution.

However, non-reversible Markov chains may have better properties in terms of mixing behavior or asymptotic variance. This can be shown experimentally in special cases \cite{Suwa2010,TuritsynChertkovVucelja2011, Vucelja2014}, theoretically in special cases \cite{DiaconisHolmesNeal2000,Neal2004}, and in fact, also in general \cite{Sun2010,Chen2013}, with respect to asymptotic variance. See also \cite{Rey-Bellet2014} for improved asymptotic variance of non-reversible diffusion processes on compact manifolds. 

There exist two basic approaches to the construction of non-reversible chains from reversible chains: one can `lift' the Markov chain to a larger state space \cite{DiaconisHolmesNeal2000,Neal2004,TuritsynChertkovVucelja2011,Vucelja2014}, or one can introduce non-reversibility without altering the state space \cite{Sun2010}. In continuous spaces, the \emph{hybrid} or \emph{Hamiltonian Monte Carlo} \cite{Horowitz1991, neal2011mcmc} is closely related to the lifting approach in discrete spaces.
Other noteworthy publications on non-reversible Markov chains are \cite{Wilmer1999,GeyerMira2000}. 

In this paper we consider the second type of creating non-reversibility, i.e. without augmenting the state space. In discrete spaces this can, in principle, be achieved by changing transition probabilities (see Remark~\ref{rem:adding_gamma}). However this may be have computational disadvantages, since it requires access to all transition probabilities. Furthermore, there is no such analogue in continuous spaces (crudely speaking, because all transition probabilities to specific states are zero). 

To remedy these issues, in this paper MH is extended to `non-reversible Metropolis-Hastings' (NRMH) which allows for non-reversible transitions. The main idea of this paper is to modify the acceptance ratio, which is further  discussed in Section~\ref{sec:non-reversible-MH}. For pedagogical purposes, the theory is first developed for discrete state spaces. It is shown how the acceptance probability of MH, can be adjusted so that the resulting chain in NRMH has a specified `vorticity', and therefore, will be non-reversible. Any Markov chain satisfying a symmetric structure condition can be constructed by NRMH, which establishes the generality of the algorithm.
Theoretical advantages of finite state space non-reversible chains in terms of improved asymptotic variance and large deviations estimates are briefly mentioned in Section~\ref{sec:advantages}. In particular we recall a result from \cite{Sun2010} that adding non-reversibility decreases asymptotic variance. Also we present a variation on a result by \cite{Rey-Bellet2014} on large deviations from the invariant distribution.

As mentioned above, for continuous state spaces it was so far not clear how general non-reversible Markov chains (i.e. discrete time, for arbitrary target density) could be constructed. 
One of the main advantages of NRMH is that it also applies in the setting of continuous state spaces, and thus provides a partial solution to this problem, as will be discussed and verified experimentally in Section~\ref{sec:continuous_spaces}. In particular we implement a non-reversible version of the Metropolis Adjusted Langevin Algorithm (MALA) for Gaussian multivariate target distributions.
Finally, conclusions and directions of further research are discussed in Section~\ref{sec:discussion}.

\subsection{Notation}
We will consider both finite and infinite-dimensional vectors and matrices. The constant vector with all elements equal to $1$ will be denoted by $\1$; the dimensionality of $\1$ should always be clear from the context. Similarly the identity matrix of any dimension will be denoted by $I$. For sets $V, S$, with $V \subset S$ the indicator function of $V$ is denoted by $\1_V : S \rightarrow \R$. The transpose of a matrix $A$ is denoted by $A'$. The Euclidean vector norm on $\R^n$, as well as its induced matrix norm, will be denoted by $\|\cdot\|$.
For a matrix $A \in \R^n$, the spectrum is denoted by $\sigma(A)$. The spectral bound and spectral radius of $A$ are denoted by $\mathfrak{s}(A) = \max\{ \Re \lambda: \lambda \in \sigma(A)\}$ and $\mathfrak{r}(A) = \max\{ |\lambda| :  \lambda \in \sigma(A)\}$, respectively.

\section{Metropolis-Hastings generalized to obtain non-reversible chains}
\label{sec:non-reversible-MH}
As a preliminary to non-reversible Metropolis-Hastings, we require the notion of \emph{vorticity matrix}, which is introduced in Section~\ref{sec:vorticity}. The classical Metropolis-Hastings algorithm, discussed in Section~\ref{sec:MH}, is extended using the notion of vorticity matrix to a non-reversible version in Section~\ref{sec:non-reversible_extension}.

\subsection{Non-reversible Markov chains and vorticity}
\label{sec:vorticity}
Let $P = (P(x,y))$ denote a matrix of transition probabilities of a Markov chain on a finite or countable state space $S$.
A \emph{distribution} on $S$ is a vector with positive elements in $L^1(S)$, and is not necessarily normalized, i.e. it is not necessarily the case that $\sum_{x \in S} \pi(x) = 1$. If $\pi$ is a distribution such that $\sum_{x \in S} \pi(x) = 1$, then we call $\pi$ a \emph{probability distribution}. We will always assume that $\pi(x) > 0$ for all $x \in S$.
A (probability) distribution $\pi$ on $S$ is said to be an \emph{invariant (probability) distribution} of $P$ if
$\pi' P = \pi'$, i.e. $\sum_{x \in S} \pi(x) P(x,y) = \pi(y)$ for all $y \in S$. A distribution $\pi$ on $S$ is said to satisfy the \emph{detailed balance condition} with respect to $P$ if $\diag(\pi) P = P' \diag(\pi)$, i.e. $\pi(x) P(x,y) = P(y,x) \pi(y)$ for all $x, y \in S$. If there exists a distribution $\pi$ on $S$ that satisfies detailed balance with respect to $P$, then $P$ is said to be \emph{reversible}. As is well known, and straightforward to check, if $\pi$ satisfies detailed balance with respect to $P$, then $\pi$ is invariant for $P$. A chain which does not satisfy the detailed balance condition with respect to its invariant distribution is called \emph{non-reversible}. In a certain sense, this is a misnomer: we may obtain a time reversed Markov chain $\widehat P$ by defining $\widehat P(x,y) := \frac{ \pi(y) P(y,x)}{\pi(x)}$. In fact, $\widehat P$ is the adjoint of $P$ with respect to the inner product $(\cdot,\cdot)_{\pi}$, defined by $(f,g)_{\pi} = \sum_{x \in S} f(x) g(x) \pi(x)$. For further background material on Markov chains the reader is referred to \cite{Levin2009}.

Consider a non-reversible Markov chain $P$. Let $K = \half ( P + \widehat P)$. Then $K$ is a reversible Markov chain with invariant distribution $\pi$. We will also define a vorticity matrix $\Gamma$ as (essentially) the skew-symmetric part of $P$:
\begin{equation} \label{eq:Gamma} \Gamma(x, y) := \pi(x) P(x,y) - \pi(y) P(y,x), \quad x, y \in S,\end{equation}
or in matrix notation
\[ \Gamma = \diag(\pi) P - P' \diag(\pi).\]
One can think of $\Gamma$ as (a transformation of) the skew-symmetric part of $P$: $\Gamma = \diag(\pi) (P - \widehat P)$.
The following simple observations are fundamental to this paper.
\begin{lemma} 
\label{lem:vorticity_observations}
Let $P$ be the transition matrix of a Markov chain on $S$ and let $\pi$ be a distribution on $S$. Let $\Gamma$ be defined by~\eqref{eq:Gamma}. Then:
\begin{itemize}
 \item[(i)] $\Gamma$ is skew-symmetric, i.e. $\Gamma = - \Gamma'$;
 \item[(ii)] $\pi$ satisfies detailed balance with respect to $P$ if and only if $\Gamma = 0$;
 \item[(iii)] $\pi$ is invariant for $P$ if and only if $\Gamma \1 = 0$.
\end{itemize}
\end{lemma}

\begin{proof}
(i) and (ii) are immediate. As for (iii), note that
\[ (\Gamma \1)(x) = \sum_{y \in S} \left[ \pi(x) P(x, y) - \pi(y) P(y,x)\right] = \pi(x) - \sum_{y \in S} \pi(y) P(y,x),\]
which is zero for all $x$ if and only if $\pi$ is invariant for $P$.
\end{proof}

In light of Lemma~\ref{lem:vorticity_observations}, a matrix $\Gamma \in \R^{n \times n}$ which is skew-symmetric and satisfies $\Gamma \1 = 0$ is called a \emph{vorticity matrix}. If $\Gamma$ is related by~\eqref{eq:Gamma} to a Markov chain $P$ with invariant distribution $\pi$, it is called the \emph{vorticity of $P$ and $\pi$}. It will be a key ingredient in the construction of a non-reversible version of Metropolis-Hastings. 

\begin{remark}
\label{rem:adding_gamma}
A direct way of constructing a non-reversible chain $P$ from a reversible chain $K$ and a vorticity matrix $\Gamma$, is by letting $P(x,y) = K(x,y) + \frac 1 {2 \pi(x)}\Gamma(x,y)$, provided that $P$ is a probability matrix (i.e. has nonnegative entries). This is discussed in e.g. \cite{Sun2010}.
In order to make a transition from a state $x$, one has to compute entries of $K(x,\cdot)$ and $\Gamma(x,\cdot)$ to determine the transition probabilities. This approach has no alternative in uncountable state spaces. This is the main reason for wishing to develop a method that does not depend on the construction mentioned in this remark.
\end{remark}

\subsection{Metropolis-Hastings}
\label{sec:MH}
In the Metropolis-Hastings (MH) algorithm a reversible Markov chain $P_0$ with a given invariant distribution $\pi$ is constructed. We will assume, mainly for simplicity, throughout this paper that $\pi(x) > 0$ for all $x \in S$. As an ingredient for the construction of $P_0$, a Markov chain $Q$ is used, satisfying the \emph{symmetric structure} condition
\begin{equation} \label{eq:symmetric_structure} Q (y,x) = 0 \quad \mbox{whenever} \quad Q(x,y) = 0, \quad x, y \in S.\end{equation} 
In other words, whenever a transition from $x$ to $y$ has positive probability, the reverse probability also has positive probability. 
The \emph{Hastings ratio} $R_0(x,y)$ is defined as
\begin{equation} \label{eq:hastings-ratio-MH}
 R_0(x,y) := \left\{ \begin{array}{ll} \frac{\pi(y) Q(y,x)}{\pi(x) Q(x,y)}, \quad & \mbox{for all} \ x, y \in S \ \mbox{for which $\pi(x) Q(x,y) \neq 0$}, \\
                      1 \quad & \mbox{otherwise.}
                     \end{array} \right.
\end{equation}
With this definition of $R_0$, acceptance probabilities are defined as
\begin{equation} \label{eq:acceptance-MH} 
A_0(x,y) := \min \left( 1,  R_0(x,y) \right),
\end{equation}
and transition probabilities $P_0(x,y)$ are defined by
\begin{equation} \label{eq:transition-MH}
P_0(x,y) := \left\{\begin{array}{ll} Q(x,y) A_0(x,y), \quad & x \neq y, \\
                    Q(x,x) + \sum_{z \neq x} Q(x,z) (1 - A_0(x,z)), \quad & x = y.
                   \end{array} \right.
\end{equation}
It is a straightforward exercise to show that the chain $P_0$ has $\pi$ as its invariant distribution. An important step is the observation that $R_0(x,y) \leq 1$ if and only if $R_0(y,x) \geq 1$, which will be a recurring phenomenon in the sequel.

\subsection{Non-reversible Metropolis-Hastings}
\label{sec:non-reversible_extension}

We will now discuss how this framework can be extended to construct Markov chains that are, in general, non-reversible. Let $\Gamma \in \R^{n \times n}$ be a vorticity matrix, and let $Q$ be the transition matrix of a Markov chain, satisfying~\eqref{eq:symmetric_structure}. Again, $\pi : S \rightarrow (0,\infty)$ is some distribution that is not necessarily normalized and has only positive entries.

We define for $x, y \in S$ the \emph{non-reversible Hastings ratio} as
\begin{equation}
\label{eq:hastings-ratio}
R_{\Gamma}(x,y) := \left\{ \begin{array}{ll} \frac{\Gamma(x,y) + \pi(y) Q(y,x)}{\pi(x) Q(x,y)}, \quad & \mbox{if} \ \pi(x) Q(x,y) \neq 0, \\
                            1 \quad &\mbox{otherwise,}
                           \end{array} \right.
\end{equation}
and let, analogously to MH, the acceptance probabilities $A_{\Gamma}$ be
\begin{equation} \label{eq:acceptance-Gamma}
A_{\Gamma}(x,y) := \min \left( 1, R_{\Gamma}(x,y) \right).
\end{equation}
Entries of $\Gamma$ can be negative. In order to avoid the situation that $A_{\Gamma}$ becomes negative, we will explicitly constrain vorticity matrix $\Gamma$ to satisfy
\begin{equation} \label{eq:vorticity_bound} \Gamma(x,y) \geq - \pi(y) Q(y,x) \quad \mbox{for all} \ x, y \in S.\end{equation}
Note that~\eqref{eq:vorticity_bound} implies, by skew-symmetry of $\Gamma$, that
\[ - \pi(y) Q(y,x) \leq \Gamma(x,y) \leq \pi(x) Q(x,y) \quad \mbox{for all} \ x, y \in S.\]
In particular, by the symmetric structure condition~\eqref{eq:symmetric_structure}, $\Gamma$ should have zeroes wherever $Q$ has zeroes.
As with Metropolis-Hastings, the transition probabilities $P_{\Gamma}(x,y)$ are defined by
\begin{equation}
\label{eq:transition-Gamma}
P_{\Gamma}(x,y) := \left\{\begin{array}{ll} Q(x,y) A_{\Gamma}(x,y), \quad & x \neq y, \\
                    Q(x,x) +\sum_{z \neq x} Q(x,z) (1 - A_{\Gamma}(x,z)), \quad & x = y.
                   \end{array} \right.
\end{equation}
Note that indeed $P_{\Gamma}$ is a matrix of transition probabilities.
For $\Gamma = 0$, $A_{\Gamma}$ and therefore $P_{\Gamma}$ reduce to $A_0$ and $P_0$, so that the chosen notation is consistent.

In order to check that the proposed Markov chain has $\pi$ as its invariant density, we need to verify that $\Gamma$, $\pi$ and $P_{\Gamma}$ are related through~\eqref{eq:Gamma}. As a crucial step, we employ the following lemma, in analogy with Metropolis-Hastings.

\begin{lemma}
\label{lem:hastings-ratio}
Let $\Gamma$ be a vorticity matrix, $Q$ a matrix of transition probabilities satisfying~\eqref{eq:symmetric_structure}, $\pi$ a distribution that is nowhere zero, such that~\eqref{eq:vorticity_bound} holds. Let $R_{\Gamma}$ be as above. Then 
$R_{\Gamma}(y,x) > 1$  if and only if $R_{\Gamma}(x,y) < 1$
for any $x,y \in S$ for which $Q(x,y) \neq 0$.
\end{lemma}
\begin{proof}
Suppose $R_{\Gamma}(x,y) < 1$, i.e.  $\Gamma(x,y) + \pi(y) Q(y,x) < \pi(x) Q(x,y)$.
Then
\[ \pi(x) Q(x,y) + \Gamma(y,x) = \pi(x) Q(x,y) - \Gamma(x,y) > \pi(y) Q(y,x),\]
so that $R_{\Gamma}(y,x) > 1$, using that $\Gamma$ is skew-symmetric. The converse direction is analogous.
\end{proof}

Using the previous lemma, it is now straightforward to show that $\Gamma$ is the vorticity matrix of $(P_{\Gamma},\pi)$.

\begin{lemma}
\label{lem:P_Gamma_has_vorticity_Gamma}
Let $Q$ be a Markov chain, $\Gamma$ a vorticity matrix, and $\pi$ a distribution on $S$, such that~\eqref{eq:symmetric_structure} and \eqref{eq:vorticity_bound} are satisfied. Let $P_{\Gamma}$ be defined through~\eqref{eq:hastings-ratio}, \eqref{eq:acceptance-Gamma} and \eqref{eq:transition-Gamma}. Then~\eqref{eq:Gamma} holds for $(P_{\Gamma}, \pi)$, i.e.
\[ \Gamma(x, y) = \pi(x) P_{\Gamma}(x,y) - \pi(y) P_{\Gamma}(y,x), \quad \mbox{for all} \ x, y \in S.\]
\end{lemma}

\begin{proof}
If $x$, $y$ are such that $R_{\Gamma}(x,y) < 1$. Then $A_{\Gamma}(x,y) = R_{\Gamma}(x,y)$, and $A_{\Gamma}(y,x) = 1$. Therefore
\begin{align*} \pi(x) P_{\Gamma}(x,y) & = \pi(x) Q(x,y) A_{\Gamma}(x,y) = \pi(x) Q(x,y) \left( \frac{\Gamma(x,y) + \pi(y) Q(y,x)}{\pi(x) Q(x,y)} \right)\\
& = \Gamma(x,y) + \pi(y) Q(y,x) = \Gamma(x,y) + \pi(y) A_{\Gamma}(y,x) Q(y,x) = \Gamma(x,y) + \pi(y) P_{\Gamma}(y,x),\end{align*}
using Lemma~\ref{lem:hastings-ratio}.
so that~\eqref{eq:Gamma} is satisfied for such $x$ and $y$.
The cases in which $R_{\Gamma}(x,y) = 1$ or $> 1$ are analogous.
\end{proof}

Finally, since by the assumption that $\Gamma$ is a vorticity matrix, $\Gamma \1 = 0$, and hence Lemma~\ref{lem:vorticity_observations} (iii) gives that $\pi$ is invariant for $P_{\Gamma}$. We have obtained our main result.

\begin{theorem}
Let $Q$ be a Markov chain, $\Gamma$ a vorticity matrix, and $\pi$ a distribution on $S$ that is everywhere positive, such that~\eqref{eq:symmetric_structure} and \eqref{eq:vorticity_bound} are satisfied. Let $P_{\Gamma}$ be defined through~\eqref{eq:hastings-ratio},~\eqref{eq:acceptance-Gamma} and \eqref{eq:transition-Gamma}. Then $P_{\Gamma}$ has $\pi$ as invariant distribution and $\Gamma$ as its vorticity matrix.
\end{theorem}

\begin{remark}
We will refer to a combination $(Q, \Gamma,\pi)$ which satisfies~\eqref{eq:symmetric_structure} and~\eqref{eq:vorticity_bound} as a \emph{compatible combination}.
Especially verifying condition~\eqref{eq:vorticity_bound} requires some knowledge about $\pi$, but these do not need to be exact: it will suffice to have acces to a lower bound for $\pi$. In the proof of Theorem~\ref{thm:ornstein_uhlenbeck} the analogue of this condition for continuous spaces is checked as an example.
\end{remark}

\begin{remark}
Once we have access to a compatible combination of proposal chain $Q$, vorticity matrix $\Gamma$ and target distribution $\pi$, the NRMH algorithm has similar favorable properties as MH, in that only local information is required: $Q$, $\pi$ and $\Gamma$ only need to be evaluated at the current and proposed state, and no normalization of $\pi$ is required. In Section~\ref{sec:continuous_spaces} this will become even more important when NRMH is applied to a problem in continuous state space.
\end{remark}

\subsection{General observations on NRMH}

The following trivial observation serves to indicate the generality of this approach. It asserts that every Markov chain may be build, in a trivial way, by the described procedure.

\begin{proposition}
\label{prop:very_general}
Let $P$ be a Markov chain with invariant distribution $\pi$ and corresponding vorticity $\Gamma$, satisfying~\eqref{eq:symmetric_structure}. 
If we use $Q = P$ as proposal distribution and $\Gamma$ as vorticity matrix in the NRMH algorithm with target distribution $\pi$, then the resulting Markov chain $P_{\Gamma}$ is identical to $P$.
\end{proposition}

\begin{proof}
It suffices to note, that by $Q = P$ and~\eqref{eq:Gamma}, $A_{\Gamma}(x,y) = 1$ for all $x, y \in S$, $x \neq y$.
\end{proof}

\begin{remark}
If, for some pair $(x, y) \in S \times S$,~\eqref{eq:vorticity_bound} holds with equality, the transition probability $P_{\Gamma}(x,y) = 0$ even when $Q(x,y) \neq 0$. Therefore irreducibility of $Q$ does not imply irreducibility of $P_{\Gamma}$, unless we impose the stronger condition:
\begin{equation} \label{eq:irreducibility_bound} \Gamma(x,y) > - \pi(y) Q(y,x) \quad \mbox{for all} \ x, y \in S.\end{equation}
\end{remark}

As noted in Remark~\ref{rem:adding_gamma}, one may alternatively construct any non-reversible chain $P$ by `adding' a vorticity matrix to a reversible chain $K$. We may translate one approach into the other, as follows:

\begin{proposition}
 \label{prop:equivalent_approaches}
Consider irreducible transition kernels $Q(x,y)$ and $H(x,y)$, related by $Q(x,y) = H(x,y) + \frac 1 {2 \pi(x)} \Gamma(x,y)$, both satisfying the symmetric structure condition~\eqref{eq:symmetric_structure}. Suppose $Q$, $\pi$ and $\Gamma$ satisfy~\eqref{eq:vorticity_bound}.
Suppose $P_1$ is the Markov kernel obtained from NRMH, using $Q$ as proposal chain and $\Gamma$ as vorticity matrix.
Suppose $P_2$ is the Markov kernel given by $P_2(x,y) = K(x,y) + \frac 1{2\pi(x)} \Gamma(x,y)$, where $K$ is the classical Metropolis-Hastings kernel obtained by using $H$ as proposal chain (and $\pi$ as target distribution). Then $P_1 = P_2$.
\end{proposition}

\begin{proof}
We compute, for $x \neq y$, with $Q(x,y) \neq 0$,
\begin{align*} \pi(x) P_1(x,y) & = \pi(x) \min \left(1, \frac{ \Gamma(x,y) + \pi(y) Q(y,x)}{\pi(x) Q(x,y)}  \right) Q(x,y) \\
& = \min \left( \pi(x) Q(x,y), \Gamma(x,y) + \pi(y) Q(y,x) \right)  \\
& = \min \left( \pi(x) H(x,y) + {\frac 1 2} \Gamma(x,y), \Gamma(x,y) + \pi(y) H(y,x) + {\frac 1 2} \Gamma(y,x) \right) \\
& = \min \left( \pi(x) H(x,y), \pi(y) H(y,x) \right) + {\frac 1 2} \Gamma(x,y) \\
& = \pi(x) \min \left( 1, \frac{\pi(y) H(y,x)}{\pi(x) H(x,y)}\right) H(x,y) + {\frac 1 2} \Gamma(x,y) = \pi(x) P_2(x,y).
\end{align*}
Both $P_1$ and $P_2$ have zeros on the off-diagonal entries for which $Q$ has zeros. Since both $P_1$ and $P_2$ represent transition probabilities, this also fixes their diagonal elements, which concludes the proof.
\end{proof}

\section{Advantages of non-reversible Markov chains in finite state spaces}
\label{sec:advantages}

Non-reversible Markov chains may offer important computational advantages compared to reversible chains. We will briefly discuss some of these advantages as they apply to finite state space Markov chains. It is to our knowledge an open question how to extend the results of Sections~\ref{sec:asymptotic_variance} and Sections~\ref{sec:large_deviations} in a generic way to uncountable and continuous state spaces.

\subsection{Asymptotic variance}
\label{sec:asymptotic_variance}
Consider a Markov chain $P$ on $S$ with invariant \emph{probability} distribution $\mu$.
Let $f : S \rightarrow \R$. We say that $f$ satisfies a Central Limit Theorem (CLT) if there is a $\sigma^2_f < \infty$ such that the normalized sum $n^{-1/2} \sum_{i=1}^n [f(X_i) - \mu(f)]$ converges weakly to a $N(0,\sigma_f^2)$ distribution. In this case $\sigma_f^2$ is called the asymptotic variance.

We will in this section work under the assumption that $S$ is finite. In this case, for any $f : S\rightarrow \R$ and irreducible $P$, a CLT is satisfied (see e.g. \cite{RobertsRosenthal2004}). The following result is obtained in \cite{Sun2010}, for a more extensive argument see \cite{Chen2013}.

\begin{proposition}
\label{prop:vorticity_improves_asymptotic_variance}
Let $K$ be a transition matrix of an irreducible reversible Markov chain with invariant probability distribution $\mu$. Let $\Gamma$ be a non-zero vorticity matrix and let $P = K + {\frac 1 2} \diag(\mu)^{-1} \Gamma$ be the transition matrix of an irreducible Markov chain. For any $f : S \rightarrow \R$ and denote by $\sigma_{f,K}^2$ and $\sigma_{f,P}^2$ the asymptotic variances of $f$ with respect to the transition matrices $K$ and $P$, respectively. 

Then for all $f : S \rightarrow \R$, we have $\sigma_{f,P}^2 \leq \sigma_{f,K}^2$, and there exists an $f$ such that $\sigma_{f,P}^2 < \sigma_{f,K}^2$.
\end{proposition}

In words, adding non-reversibility decreases asymptotic variance.

\subsection{Large deviations}
\label{sec:large_deviations}
In \cite{Rey-Bellet2014} it is noted that non-reversible diffusions on compact manifolds have favorable properties in terms of large deviations of the occupation measure from the invariant distribution. Inspired by their result, we present a simple (but to our knowledge novel) result in the same direction for finite state spaces.

As in the previous section, assume $S$ is finite.
We may transform a discrete time Markov chain on $S$ into a continuous time chain by making subsequent transitions after random waiting times that have independent $\mathrm{Exp}(\lambda)$ distributions. The discrete chain with transition matrix $P$ then transforms into a continuous time chain with generator 
\[ G(x,y)  = \left \{ \begin{array}{ll} \lambda P(x,y) \quad & \mbox{if $x \neq y$}, \\
                       -\lambda  \sum_{z \neq x} P(x,z) \quad & \mbox{if $x = y$}.
                      \end{array} \right.\]
The occupation measure of the resulting Markov process is defined as $L_t = \frac 1 t \int_0^t \delta_{X_s} \ d s$.
If $G$ is irreducible, the occupation measure satisfies for large time a large deviation principle with rate function
\[ I_G(\mu) = \sup_{u > 0} \left( - \sum_{x \in S} \mu_i \frac{ (G u)(x)}{u(x)} \right), \quad \mu \in \mathcal P(S),\] 
i.e. informally, for large $t$, for $A \subset \mathcal P(S)$,
\[ \P (L_t \in A) \approx \exp \left( - t \inf_{\mu \in A} I_G(\mu) \right).\]
Here $\mathcal P(S)$ is the set of probability distributions on $S$. The rate function satisfies the properties that (i) $I_G \geq 0$, (ii) $I_G$ is strictly convex, and (iii) $I_G(\mu) = 0$ if and only if $\mu = \pi$, where $\pi$ is the invariant distribution of $G$. Hence $I_G$ quantifies the probability of a large deviation of the occupation measure from the invariant distribution for large $t$. See \cite{Hollander2000} for further details.

The following result shows that deviations for non-reversible continuous time chains from the invariant distribution are asymptotically less likely than for the corresponding reversible chain.

\begin{proposition}
Suppose $G$ admits a decomposition of the form $G(x,y) = K(x,y) + \frac 1 {2 \pi(x)} \Gamma(x,y)$, where $\pi(x) K(x,y) = \pi(y) K(y,x)$ and $\Gamma(x,y) = -\Gamma(y,x)$ for all $x \neq y$.
Then $I_G(\mu) \geq I_K(\mu)$ for all $\mu \in \mathcal P(X)$, and the inequality is strict if $\Gamma u^{\star} \neq 0$.
\end{proposition}

\begin{proof}
Assume $\mu(x) > 0$ for all $x$.
By writing $\widetilde K(x,y) = \pi(x) K(x,y)$, $\widetilde \mu(x) = \mu(x) / \pi(x)$, we have
 \[ \sum_{x=1}^n \mu(x) \frac{ (G u)(x)}{u(x)} = \sum_{x=1}^n \frac{\widetilde \mu(x)}{u(x)} \left[ (\widetilde K u)(x) + \half (\Gamma u)(x) \right].\] Hence it suffices to prove the result for $G = K + \half \Gamma$, with $K$ symmetric and $\Gamma$ anti-symmetric.
 
Now take $u(x) = \sqrt{\mu(x)}$. Then
\begin{align*}
- \sum_{x=1}^n \mu(x) \frac{ (G u)(x)}{u(x)} & =  -\sum_{x,y=1}^n \sqrt{\mu(x)} K(x,y) \sqrt{\mu(y)} - \half \sum_{x,y=1}^n \sqrt{\mu(x)} \Gamma(x,y) \sqrt{\mu(y)}.
\end{align*}
The first term is equal to the 
the rate function $I_K(\mu)$ for the empirical measure in the symmetric case, see e.g. \cite[Theorem IV.14]{Hollander2000}. Since $\Gamma$ is skew-symmetric, the second term vanishes. 
Hence taking the supremum over $u$ gives a value larger than or equal to $I_K(\mu)$.

Let $u^{\star}$ be given by $u^{\star}(x) = \sqrt{\mu(x)/\pi(x)}$. Then
\[ \left. \nabla_u \left( - \sum_{x=1}^n \mu(x) \left( \frac{\frac 1 {2 \pi(x)} \sum_{y=1}^n \Gamma(x,y) u(y) + \sum(y) K(x,y) u(y)}{u(x)} \right) \right)  \right|_{u^{\star}} = \Gamma u^{\star},\]
which proves the second statement in the proposition.

If $\mu(x) = 0$ for some $x$, the proof carries over by only summing over indices for which $\mu(x) > 0$.
\end{proof}

\subsection{Mixing time and spectral gap}

Non-reversibility in a Markov chain can have very favorable effects on mixing time and spectral gap, but we are not aware of general results in this direction. The reader is referred to \cite{Chen1999, DiaconisHolmesNeal2000, Levin2009} for theoretical results in this direction, and to e.g. \cite{Sun2010, TuritsynChertkovVucelja2011, Vucelja2014} for less rigorous and/or numerical results. In these experimental results, non-reversibility especially seems to improve sampling in the case of sampling from a multimodal distribution.

\section{NRMH in Euclidean space}
\label{sec:continuous_spaces}
In this section we explain how to extend the idea of non-reversible Metropolis-Hastings algorithm to a Euclidean state space. In Appendix~\ref{app:general_state_space} it is discussed how Metropolis-Hastings can be applied in the case of a general measurable space, and the discussion in this section is a special case of this.

\subsection{General setting}
\label{sec:continuous_theory}
Suppose $Q$ is a Markov transition kernel on $\R^n$ which has density $q(x,y)$ with respect to Lebesgue measure, i.e. $Q(x, dy) = q(x,y) d y$. We will be interested in sampling from a distribution on $\R^n$ with Lebesgue density $\pi$. We do not require that $\int_{\R^n} \pi(x) \ d x = 1$.
Let $\gamma : \R^n \times \R^n \rightarrow \R$ be Lebesgue measurable and furthermore suppose $\gamma$ satisfies
\begin{equation}
\label{eq:vorticity_kernel_condition_1} \gamma(x,y) = - \gamma(y,x) \quad \mbox{for all} \ x,y \in \R^n,
\end{equation}
and
\begin{equation}
\label{eq:vorticity_kernel_condition_2} \int_{A \times \R^n} \gamma(x, y) \  d x \, dy = 0, \quad \mbox{for all} \quad A \in \mathcal B(\R^n).
\end{equation}
Here $\mathcal B(\R^n)$ denotes the Borel $\sigma$-algebra generated by open sets in $\R^n$.
Furthermore suppose that
\begin{equation}
\label{eq:nonsingularity_condition_1}
\gamma(x,y) = 0 \quad \mbox{for all} \ x, y \in \R^n \ \mbox{for which} \ \pi(x) q(x,y) = 0,
\end{equation}
\begin{equation}
\label{eq:nonsingularity_condition_2}
\pi(x)q(x,y) = 0 \quad \mbox{if and only if} \quad  \pi(y) q(y,x) = 0, \quad \mbox{for all}\ x, y \in \R^n,
\end{equation}
and
\begin{equation}
\label{eq:nonnegativity_condition}\gamma(x,y) + \pi(y) q(y,x) \geq 0, \quad \mbox{for all} \quad x, y \in \R^n \ \mbox{for which} \ \pi(x) q(x,y) \neq 0.
\end{equation}

Define the Hastings ratio 
\begin{equation} \label{eq:hastings_ratio_regular} R(x,y) := \left\{ \begin{array}{ll} \frac{\gamma(x,y) + \pi(y) q(y,x)}{\pi(x) q(x,y)}, \quad & \pi(x) q(x,y) \neq 0, \\
1, \quad & \pi(x)q(x,y) = 0,\end{array} \right. \quad (x, y \in \R^n),
\end{equation}
acceptance probabilities $A(x,y) := \min(1, R(x,y))$,
and transition kernel
\[
P(x,B) := \int_B A(x,y) q(x,y) \ d  y + \left( 1 - \int_S A(x,y) q(x,y) \ d y \right) \1_{x \in B}. 
\]

The proof of the following theorem can be found in Appendix~\ref{app:general_state_space}. 
\begin{theorem}
\label{thm:regular_case}
With the above notation and definitions, and assuming conditions~\eqref{eq:vorticity_kernel_condition_1}, \eqref{eq:vorticity_kernel_condition_2}, \eqref{eq:nonsingularity_condition_1}, \eqref{eq:nonsingularity_condition_2} and~\eqref{eq:nonnegativity_condition} are satisfied, $P$ is a Markov transition kernel with invariant density $\pi$.
\end{theorem}

In analogy with the discrete state space setting we call $\gamma$ the \emph{vorticity density} of $(P,\pi)$. If $\gamma \neq 0$ on a set of positive Lebesgue measure then $P$ is non-reversible, i.e. there exist sets $B_1, B_2 \subset \R^n$ such that
\[ \int_{B_1} \left\{ \int_{B_2} \pi(x) P(x,dy)\right\} \ d x \neq \int_{B_2} \left\{ \int_{B_1} \pi(x) P(x, dy) \right\} \  dx.\]

\subsection{Langevin diffusions for sampling in Euclidean space}

The application of non-reversible sampling methods in Euclidean space is a relatively unexplored area. In this short review section we will discuss the use of Langevin diffusions for simulating from a target density, and discuss the potential role of NRMH in within this context.

Assume that the target density $\pi$ is continuously differentiable. It is well known that $\pi$ is invariant for the \emph{Langevin diffusion}
\begin{equation} \label{eq:langevin} d X(t) = \nabla (\log \pi)(X(t)) \ d t + \sqrt{2} \ d W(t), \quad t \geq 0.\end{equation}
where $W$ is a standard Brownian motion in $\R^n$. A natural (discrete time) Markov chain for sampling $\pi$ is then the \emph{Euler-Maruyama discretization} of the Langevin diffusion,
\begin{equation} \label{eq:euler-maruyama} X_{k+1} \sim N(X_k + h \nabla (\log \pi)(X_n), 2 h), \quad k = 0, 1, 2, \hdots.\end{equation}
Here $h$ is a suitable stepsize. This discretization is approximately correct if $h$ is chosen to be sufficiently small. However, such a choice of $h$ results in slow convergence to equilibrium of the Markov chain. When $h$ is large, then the discretization results in large discrepancy between~\eqref{eq:langevin} and~\eqref{eq:euler-maruyama}. As a result also the respective invariant distributions will be different and in particular the invariant distribution of~\eqref{eq:euler-maruyama} will not correspond to the desired distribution with density $\pi$.
It is customary to correct for this by considering the Euler-Maruyama discretization as a proposal for Metropolis-Hastings, resulting in the Metropolis Adjusted Langevin Algorithm (MALA, \cite{RobertsTweedie1996, RobertsRosenthal1998}).

Any diffusion of the form
\begin{equation} \label{eq:nonreversible-langevin} d X(t) = -(I + S) \nabla (\log \pi)(X(t)) \ d t + \sqrt{2} \, d W(t), \quad t \geq 0,\end{equation}
with $S \in \R^{n \times n}$ skew-symmetric, has $\pi$ as invariant density. If $S \neq 0$ then the diffusion is non-reversible. In analogy with Section~\ref{sec:advantages} one may hope that such a non-reversible diffusion has advantages compared to the reversible Langevin diffusion. In fact for multivariate Gaussian target distributions these advantages are clear \cite{Hwang1993, Lelievre2013}, as we will discuss below. More generally\footnote{Strictly speaking, the analysis of \cite{Rey-Bellet2014} applies to diffusions on a compact manifold.} the probability of large deviations of the empirical distribution from the invariant distribution are reduced for non-reversible diffusions \cite{Rey-Bellet2014}.

When discretizing~\eqref{eq:nonreversible-langevin} it is again necessary to correct for discretization error by a MH accept/reject step. However, since MH generates reversible chains, one should expect that also the favourable properties of non-reversibility are destroyed.
Instead, an implementation of NRMH should be able to preserve these favourable properties of non-reversible diffusions. We will illustrate this for multivariate Gaussian distributions.

\subsection{Non-reversible Metropolis-Hastings for sampling multivarite Gaussian distributions}
\label{sec:nrmh_gaussian}
Consider as target distribution a centered normal distribution with positive definite covariance matrix $V$. In this case the Langevin diffusion becomes the Ornstein-Uhlenbeck process
\begin{equation} \label{eq:reversible_ou} d X(t) = -V^{-1} X(t) \, d t + \sqrt{2} \, d W(t),\end{equation}
where $(W(t))$ is an $n$-dimensional standard Brownian motion. 
In \cite{Hwang1993}, it is shown that adding a `nonreversible' component of the form $-S V^{-1}$ to the drift, with $S$ skew-symmetric, can improve convergence of the sample covariance. 
Therefore we will instead consider the Ornstein-Uhlenbeck process with modified drift
\begin{equation} \label{eq:nonreversible_ou} d X(t) = B X(t) \ d t + \sqrt{2} \ d W(t),\end{equation}
where $B := -(I+S)V^{-1}$ with $S$ skew-symmetric. For any choice of skew-symmetric $S$, this diffusion keeps $\pi$ invariant. The convergence to equilibrium of the diffusion is governed by the spectral bound, ${\mathfrak s}(B) := \max \{ \Re \lambda : \lambda \in \sigma(B)\}$.
More specifically,
\[ \Cov(X(t)) = 2 \int_0^t e^{B s} e^{B's} \ d s \rightarrow 2 \int_0^{\infty} e^{Bs} e^{B's} \ d s = V \quad \mbox{(as $t \rightarrow \infty$)},\]
with rate of convergence
\[ \frac 1 t \ln \left\|\int_t^{\infty} e^{Bs} e^{B's} \ d s \right\| \rightarrow 2 {\mathfrak s}(B).\]
Also, ${\mathfrak s}(B) \leq {\mathfrak s}(-V^{-1})$ for any choice of $S$. In other words, adding a non-reversible term increases the speed of convergence to equilibrium. 
In \cite{Lelievre2013}, it is established that it is possible to choose $S$ optimally, resulting in ${\mathfrak s}(B) = - \operatorname{tr} (V^{-1}) / n$. By choosing $S$ in such a way the convergence of the chain is effectively governed by the average of the eigenvalues, which should be compared to the reversible case in which the `worst' eigenvalue determines the speed of convergence.

We will apply the theory developed in Section~\ref{sec:continuous_theory} to the time discretization of the non-reversible Ornstein-Uhlenbeck process. To be able to satisfy~\eqref{eq:nonnegativity_condition} later on, we will require flexibility in the magnitude of the drift multiplier $B$ and the diffusivity. 
We consider the time discretization of~\eqref{eq:nonreversible_ou}, with step size $h > 0$, is
\begin{equation} \label{eq:nonreversible_ou_discretization} X_{k+1} = X_k + h B X_k + \sqrt{2h \sigma} Z_{k+1},\end{equation}
where $(Z_k)$ are i.i.d. standard normal and $\sigma > 0$. For $\sigma = 1$ this is the usual Euler-Maruyama discretization. The transition kernel of the Euler-Maruyama discretization will serve as our proposal distribution $Q(x,dy) = q(x,y) \, d y$, and we will first determine the vorticity density $\gamma$ of $Q$.

Let $\mathfrak{r}(A) = \max \{ |\lambda| : \lambda \in \sigma(A)\}$ denote the spectral radius of a square matrix $A$. Provided $\mathfrak{r}(I+h B) < 1$, the invariant probability distribution of $Q$ is the centered normal distribution with covariance $R$, where $R$ is the unique positive definite matrix solution to the 
\emph{discrete time Lyapunov equation} (see e.g. \cite[Theorem 13.2.1]{LancasterTismenetsky1985})
\begin{equation} \label{eq:discrete_lyapunov} R = 2 h \sigma^2  I + (I+h B) R(I+h B').
\end{equation}
Let $\rho$ denote the density of the invariant probability distribution of $Q$. Let $f(x,y) = \rho(x) q(x,y)$. The vorticity density of the proposal chain is 
\begin{equation} \label{eq:gamma} \gamma(x,y) := f(x,y) - f(y,x).\end{equation} 
As target density we have
\[ \pi(x) = \left( (2 \pi)^n \det V\right)^{-1/2} \exp\left( - \half x' V^{-1} x \right), \quad x \in \R^n.\]
It is clear that $\gamma$ satisfies~\eqref{eq:vorticity_kernel_condition_1} and~\eqref{eq:vorticity_kernel_condition_2}. Also since $\pi$ and $q$ are non-degenerate,~\eqref{eq:nonsingularity_condition_1} and~\eqref{eq:nonsingularity_condition_2} are satisfied. The same statements hold trivially for scalar multiples of $\gamma$.

Verification of~\eqref{eq:nonnegativity_condition} requires more effort. We provide a sufficient condition. The proof of this result is provided in Appendix~\ref{app:proof_ornstein_uhlenbeck}.

\begin{theorem}
\label{thm:ornstein_uhlenbeck}
Define constants $0 < C_1 \leq C_2$ by
\begin{equation}
\label{eq:constants} C_1 = \|V^{-1/2} (I+S) V^{-1} (I-S) V^{1/2}\| \quad \mbox{and} \quad C_2 = \| V^{-1/2} (I+S) V^{-1/2} \|^2 \|V\|.
\end{equation}
Suppose $c > 0$, $h > 0$ and $\sigma > 0$ satisfy 
\begin{equation} \label{eq:conditions_parameters} h < \frac 2 {C_2}, \quad \sigma^2 \leq  \frac{2 - h C_2}{2 - h (C_2 - C_1)}, \quad \mbox{and}  \quad c \leq \sigma^n.
\end{equation}
Then $\widetilde \gamma(x,y) := c \gamma(x,y)$, with $\gamma$ as constructed above, satisfies~\eqref{eq:nonnegativity_condition}, and is therefore a vorticity density compatible with proposal distribution $Q(x,dy) \sim N((I + h B)x, 2 h \sigma^2 I)$ and invariant distribution $N(0,V)$.
\end{theorem}

\begin{remark}
\label{rem:choice_parameters}
How should one choose $c$, $h$ and $\sigma$? It seems reasonable to choose $\sigma^2$ equal to the maximal allowed value in~\eqref{eq:conditions_parameters}, so that the deviation from the Euler-Maruyama discretization (which has $\sigma = 1$) is minimal; i.e. let
\begin{equation}
 \sigma = \sigma(h) = \sqrt{\frac{2 - h C_2}{2 - h (C_2 - C_1)}}.
\end{equation}
To maximize the non-reversibility effects in the acceptance probability one should choose $c$ as large as possible, i.e. $c = \sigma^n(h)$.
A heuristic estimate for the scaling of the expected step size is given by the step size $h$ times the multiplicative factor in the vorticity, $c = \sigma^n(h)$ in the acceptance probability. Note that $h \sigma^n(h) = 0$ for $h = 0$ and $h = \frac 2{C_2}$. 
Maximization of $h \sigma^n(h)$ with respect to $h$ yields
\begin{equation} \label{eq:maximal_vorticity} h = \frac{2}{C_2} + \frac{(n+2)C_1}{2 C_2(C_2-C_1)} - \frac{\sqrt{(n-2)^2 C_1^2 + 8 n C_1 C_2}}{2 C_2(C_2-C_1)}\end{equation}
as long as $C_1 < C_2$ (which is the case in which $S$ and $V$ do not commute). This expression satisfies the condition $h < \frac 2 {C_2}$. A first order Taylor approximation around $1/n$ yields the simplified expression $h \approx \frac{4}{C_2 (n+2)} < \frac{2}{C_2}$. The corresponding value of $\sigma^2(h)$ is to first order equal to $\sigma^2(h) \approx 1 - \frac{2 C_1}{C_2(n+2)}$.

In case $C_1 = C_2$, the optimal value of $h$ is given by $h = \frac{4}{(n+2)C_2}$, with $\sigma^2(h) = 1 - \frac{2}{n+2}$.
\end{remark}

To summarize, a general procedure for applying non-reversible Metropolis-Hastings may be described by the steps listed in Figure~\ref{fig:steps-nrmh}.

\begin{figure}[ht]
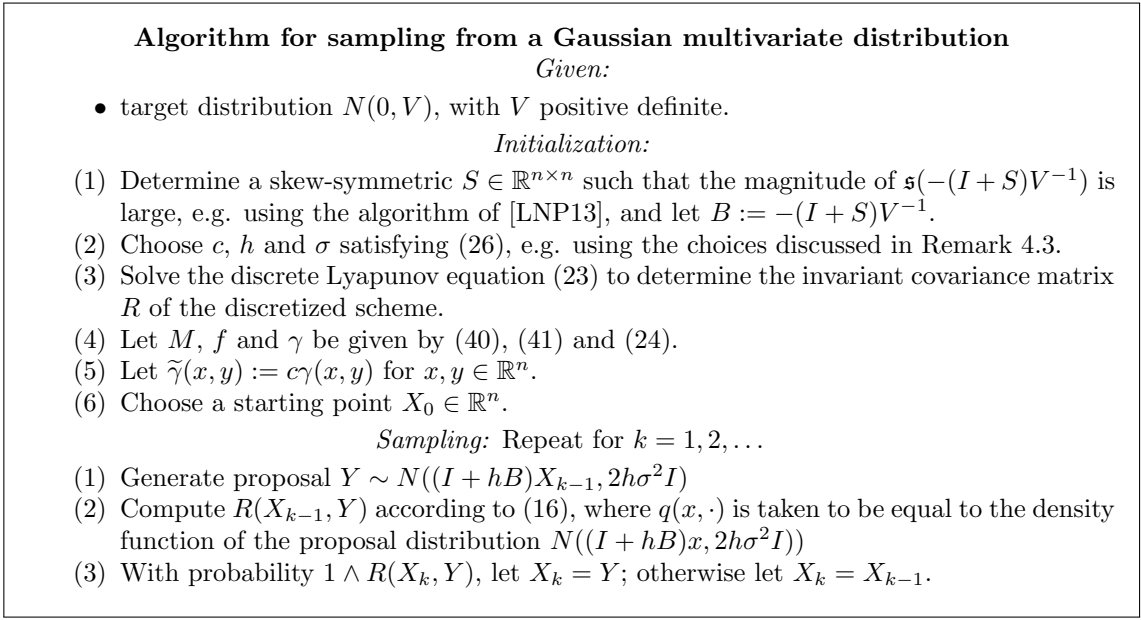

\begin{framed}
\textbf{Algorithm for sampling from a Gaussian multivariate distribution} \\

\emph{Given:}
\begin{itemize}
\item target distribution $N(0,V)$, with $V$ positive definite.
\end{itemize}
\emph{Initialization:}
\begin{enumerate}
\item Determine a skew-symmetric $S \in \R^{n \times n}$ such that the magnitude of ${\mathfrak s}(-(I+S)V^{-1})$ is large, e.g. using the algorithm of \cite{Lelievre2013}, and let $B := -(I+S)V^{-1}$.
\item Choose $c$, $h$ and $\sigma$ satisfying~\eqref{eq:conditions_parameters}, e.g. using the choices discussed in Remark~\ref{rem:choice_parameters}.
\item Solve the discrete Lyapunov equation~\eqref{eq:discrete_lyapunov} to determine the invariant covariance matrix $R$ of the discretized scheme.
\item Let $M$, $f$ and $\gamma$ be given by~\eqref{eq:M},~\eqref{eq:f} and~\eqref{eq:gamma}.
\item Let $\widetilde \gamma(x,y) := c \gamma(x,y)$ for $x, y\in \R^n$.
\item Choose a starting point $X_0 \in \R^n$.
\end{enumerate}
\emph{Sampling:} 
Repeat for $k = 1, 2, \dots$
\begin{enumerate}
 \item Generate proposal $Y \sim N((I+h B) X_{k-1}, 2h \sigma^2 I)$
 \item Compute $R(X_{k-1}, Y)$ according to~\eqref{eq:hastings_ratio_regular}, where $q(x,\cdot)$ is taken to be equal to the density function of the proposal distribution $N((I+h B)x, 2h \sigma^2 I))$
 \item With probability $1 \wedge R(X_k,Y)$, let $X_k = Y$; otherwise let $X_k = X_{k-1}$.
\end{enumerate}
\end{framed}

\caption{Non-reversible Metropolis-Hastings with Ornstein-Uhlenbeck proposals}
\label{fig:steps-nrmh}
\end{figure}
%
%

\subsection{Numerical experiments}
\label{sec:experiments_continuous}

Below we carry out two experiments illustrating the approach above. For the obtained Markov chain realization $(X_1, \dots, X_P)$, we will obtain an estimate of the decorrelation by considering the \emph{empirical autocorrelation function (EACF)} $r$ defined by 
\[ r^i(k) = \frac 1{P-k} \sum_{p=1}^{P-k} (X_p^i - \widehat \mu^i)(X_{p+k}^i - \widehat \mu^i),\]
where $i$ ranges over the coordinates $i = 1, \dots, n$, and where $\widehat \mu^i = \frac 1 P \sum_{p=1}^P X_p^i$ is the empirical average of the $i$-th coordinate. A fast decaying EACF indicates that the samples generated by the Markov chain are quickly decorrelating.

\subsubsection{Three-dimensional example}
\label{sec:3d}
In this example, from \cite{Hwang1993}, we take as target covariance structure $V$ a diagonal matrix with $\diag(V) = \begin{pmatrix} 1, 1,1/4 \end{pmatrix}$.  The optimal nonlinear drift is obtained by letting 
\[S = \begin{pmatrix} 0 & \sqrt{3} & 1 \\ -\sqrt{3} & 0 & 1 \\ -1 & -1 & 0\end{pmatrix}.\]
We choose the parameter values in accordance with Remark~\ref{rem:choice_parameters}, resulting in 
\[ c = 0.5333, \quad h = 0.0334, \quad \sigma = 0.8109.\]
The performance of NRMH is compared to MH with identical step-size $h$, and reversible proposal distribution $Q(x,dy) \sim N((I - h V^{-1})x, 2h I)$.
In Figure~\ref{fig:EACF3d} the EACFs for this 3-dimensional example are plotted. Here we see that NRMH helps to decrease the autocorrelations of the slowly decorrelating  components in MH (here, the first two components). It achieves this without increasing autocorrelations of components that are already quickly decorrelating (here, the third component).

\begin{figure}[ht]
\includegraphics[width=\textwidth]{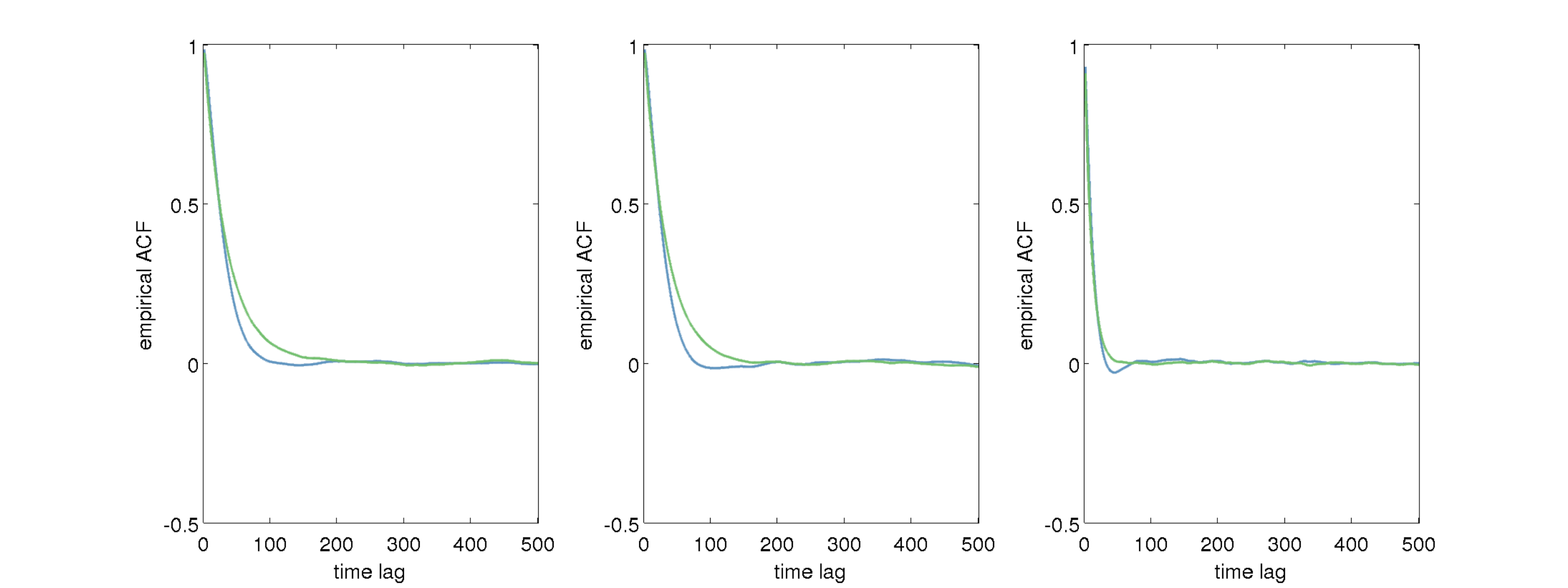}
\caption{Empirical autocorrelation functions for 3-dimensional example of Section~\ref{sec:3d}.
 The blue curve represents the non-reversible Metropolis Hastings method. 
 The green curve represents reversible proposals in conjunction with the usual MH acceptance rule.
 The step sizes $h$ for all approaches are taken to be the same. These plots are based on an MCMC trajectory of $10^6$ steps.}
\label{fig:EACF3d}
\end{figure}

\subsubsection{Nine-dimensional example}
\label{sec:9d}

Here we generated a random diagonal matrix $V$, with 
\[ \diag(V) = \begin{pmatrix} 0.8147, 0.9058, 0.1270, 0.9134, 0.6324, 0.0975, 0.2785, 0.5469, 0.9575 \end{pmatrix}.\]
Using the algorithm described in \cite{Lelievre2013} an optimal non-reversible drift $B = -(I+S)V^{-1}$ can be computed. For reversible dynamics, we have ${\mathfrak s}(-V^{-1}) = -1.0444$, while for the optimal non-reversible dynamics, ${\mathfrak s}(B) = -\operatorname{tr} V^{-1}/n = - 3.2891$.
In Figure~\ref{fig:EACF9d} the EACFs for this 9-dimensional example are plotted. One can clearly see the typical effect of adding non-reversibility: the autocorrelation of the worst coordinates is improved so that it becomes on par with that of the fastest decorrelating coordinates.

\begin{figure}[ht]
\includegraphics[width=\textwidth]{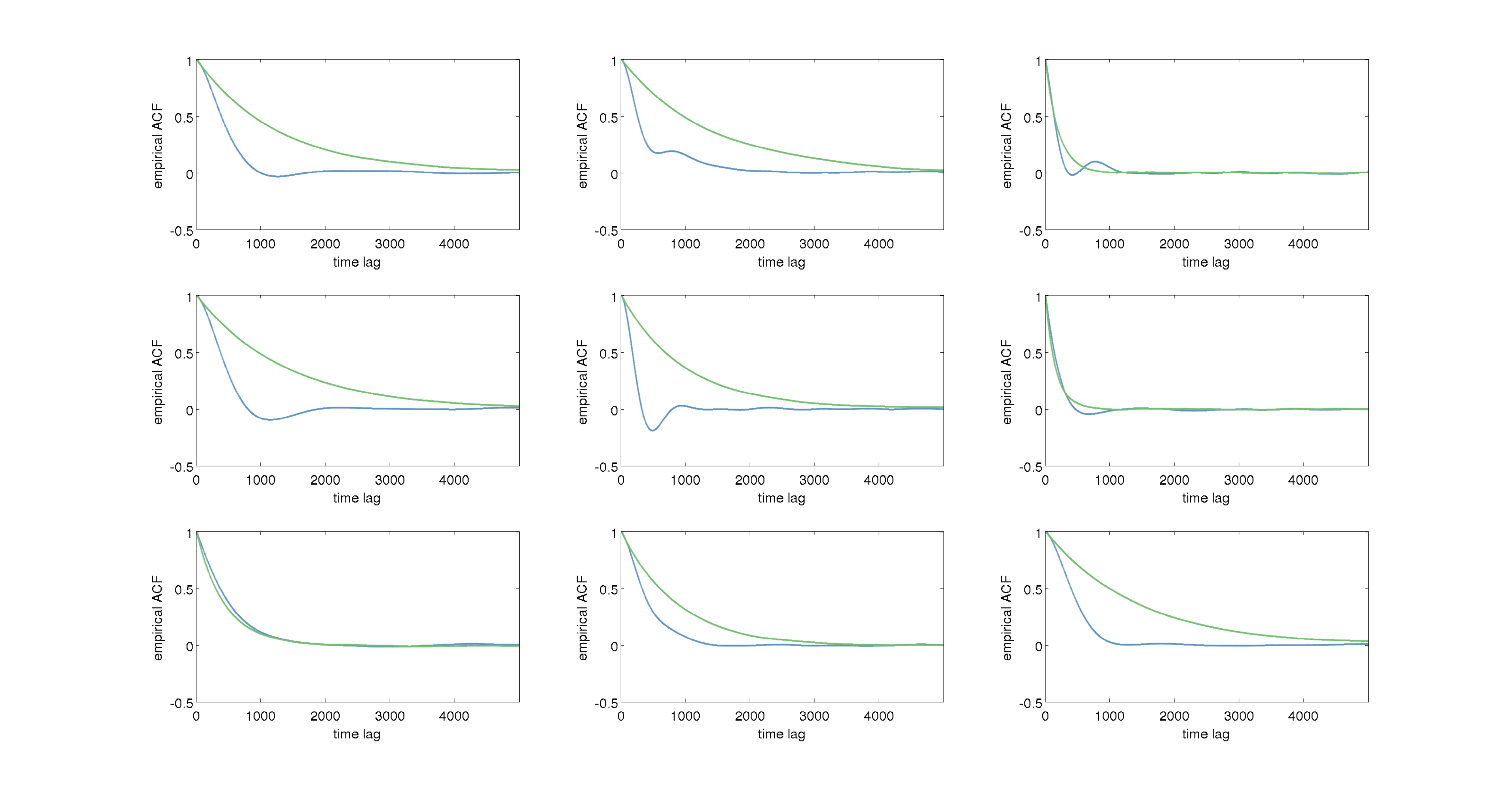}
\caption{Empirical autocorrelation functions for the 9-dimensional example of Section~\ref{sec:9d}.  The blue curve represents the non-reversible Metropolis Hastings method. 
 The green curve represents reversible proposals in conjunction with the usual MH acceptance rule.
 These plots are based on an MCMC trajectory of $10^7$ steps.}
 \label{fig:EACF9d}
\end{figure}

In this case, choosing $c$, $h$ and $\sigma$ as in Remark~\ref{rem:choice_parameters} results in values
\[ c = 0.4313, \quad h =  7.0822 \times 10^{-4} \quad \mbox{and} \quad \sigma = 0.9108.\]
In a numerical example with $10^7$ proposed transitions this leads to acceptance ratios displayed in Table~\ref{tab:acceptance_ratios}. Using the batch means method (by dividing the sample trajectory in $\sqrt{n}$ trajectories of length $\sqrt{n}$ and assuming the $\sqrt{n}$ are independent), we can estimate asymptotic variance. The resulting estimates for asymptotic variance of the different components are given in Table~\ref{tab:asvar}. It should be noted that the notion of asymptotic variance is only defined in case a CLT holds (see \cite{RobertsRosenthal2004}), which strictly speaking is an open question in this setting.

\begin{table}[ht]
 \centering
 \begin{tabular}{l|ccccccccc}
 component & 1 & 2 & 3 & 4 & 5 & 6 & 7 & 8 & 9 \\
 \hline
 NRMH &  599.96 & 661.17 &  40.80 & 572.26 & 159.05 &  27.35 & 230.41 & 401.98 & 718.64 \\
 MH & 1315.3 & 1522.2 & 47.156 & 1473.3 & 876.46 & 28.316 & 204.05 & 708.83 & 1578.2  
 \end{tabular}
 \caption{Estimated asymptotic variances for the 9-dimensional example of Section~\ref{sec:9d}, based on an MCMC trajectory of $10^7$ steps. Note that NRMH improves the (estimated) asymptotic variance for almost all components (and often significantly), except for component \#7.}
 \label{tab:asvar}
\end{table}

It is well known that for optimal convergence in Metropolis Adjusted Langevin (MALA), the stepsize should be tuned so that the ratio of accepted proposals is approximately equal to 0.574 \cite{RobertsRosenthal1998}. Compared to this, the acceptance ratios in our example, given in Table~\ref{tab:acceptance_ratios}, are fairly high. In particular the Metropolis-Hastings chain can be improved significantly by increasing step size $h$, and so we are currently comparing NRMH with a sub-optimal tuning of MH. 

The results of this experimental section should therefore be considered as a proof of concept of NRMH in continuous state spaces, rather than as an advertisement for its immediate practicality. The experiments do illustrate the faster decorrelation of NRMH in comparison to MH (for fixed step-size). It is an open question if the framework of NRMH can be extended so that NRMH would become competitive with optimally tuned MH.

\begin{table}[ht]
\centering
 \begin{tabular}{l|l}
 & acceptance ratio \\
 \hline
 NRMH & 0.7383  \\
 MH &  0.9343
 \end{tabular}
\caption{Acceptance ratios for 9-dimensional example of Section~\ref{sec:9d}, for a sample path consisting of $10^7$ proposals, and with a step size $h = 7.0822 \times 10^{-4}$.}
\label{tab:acceptance_ratios}
\end{table}

\section{Discussion}
\label{sec:discussion}

The efficiency increase of non-reversible Markov chains in MCMC can be significant, in terms of either asymptotic variance or mixing properties, as remarked in this paper. NRMH extends the MCMC-toolbox with a method to utilize these benefits. In particular for continuous state spaces it was, to our knowledge, not known how to construct non-reversible Markov chains for MCMC sampling (taking into account the necessity of a correction step when using time discretization of diffusions). 

Using the theory developed in Section~\ref{sec:continuous_spaces}, NRMH can be applied to general distributions on $\R^n$ as follows.
For a target density function $\pi$, suppose there exists a Gaussian distribution $N(0,V)$ with density function $\pi_0$ on $\R^n$, satisfying $k \pi_0 \leq \pi$ on $\R^n$ for some $k > 0$. Then if $\gamma$ is a suitable vorticity density for sampling from $N(0,V)$, using proposal density $q(x,y)$, we have for $\widetilde \gamma := k \gamma$ that
\[ \widetilde \gamma(x,y) + \pi(x) q(x,y) = k \gamma(x,y) + \pi(x)q(x,y) \geq k (\gamma(x,y) + \pi_0(x) q(x,y)) \geq 0,\]
so that~\eqref{eq:nonnegativity_condition} is satisfied for the combination for this choice of $\pi$, $\widetilde \gamma$ and $q$, and Theorem~\ref{thm:regular_case} applies. Such a suitable choice of $\gamma$ may be determined as described in Section~\ref{sec:nrmh_gaussian}.

The approach outlined in Section~\ref{sec:continuous_spaces} should be considered as a first attempt at implementing the NRMH framework for continuous spaces. As discussed, in order to use the framework one needs to verify the non-negativity condition which leads to technical challenges. In particular we expect that much progress is possible in weakening the conditions of Theorem~\ref{thm:ornstein_uhlenbeck}. The current form of that proposition results in a relatively small step size $h$, which obstructs fast convergence of NRMH. As mentioned before, it is an open question whether NRMH in continuous spaces can be made competitive with optimally tuned MALA.

The theoretical discussion of Section~\ref{sec:advantages} and the numerical experiment of Section~\ref{sec:experiments_continuous} illustrate how efficiency can be improved by employing non-reversible Metropolis-Hastings. 
In view of these encouraging results it will hopefully be possible to extend the result to more general settings. The practical application of NRMH depends on the identification of suitable vorticity structures that are compatible with proposal chains, and establishing these in practical examples provides a promising and challenging direction of research.

Analysis of non-reversible Markov chains is difficult, essentially because self-adjointness is lost. Without self-adjointness, it is much more difficult to connect spectral theory to mixing properties of chains. It seems that a good way of understanding benefits of non-reversible sampling is by studying Cesaro averages (see \cite{Levin2009} and e.g. the result on large deviations in  Section~\ref{sec:large_deviations}). The results of Section~\ref{sec:advantages} which establish that non-reversible chains have better asymptotic variance or large deviations properties, are so far qualitative in nature (i.e. fail to quantify the amount of improvement). To obtain quantitative results is an important challenge that remains to be addressed. Also, it is object of further study how these results carry over to countable and uncountable state spaces. In particular, the question under what conditions the resulting chains are geometrically ergodic and/or satisfy a CLT should be considered.

\subsection*{Acknowledgements}
I am grateful to Prof. Pavliotis (Imperial College, London) for making available the code for computing optimal non-reversible drift (as discussed in \cite{Lelievre2013}). I also wish to acknowledge valuable discussions with Prof. Hilbert J. Kappen (Radboud University), Dr. Kevin Sharp (University of Oxford) and Prof. Gareth Roberts (University of Warwick).

We thank the reviewers and editor for their valuable suggestions which have had a significant impact upon the paper.

\appendix

\section{NRMH in general state spaces}
\label{app:general_state_space}

Let $(S, \mathcal S)$ be a measurable space.
Let $P(x,dy)$ denote a Markov transition kernel and $\pi$ an invariant probability distribution of $P$, i.e. $\int_S P(x, A) \pi(dx) = \pi(A)$ for $A \in \mathcal S$. Define $F_P(dx,dy) := \pi(dx) P(x, dy)$, $B_P(dx,dy) = \pi(dy) P(y, dx)$. Here $F$ and $B$ in $F_P$ and $B_P$ denote `forward' and `backward', respectively. Note that $F_P$ and $B_P$ are probability measures on $S \times S$ with marginal distributions $\pi$.

Define 
\begin{equation} \label{eq:vorticity_measure} \Gamma(dx, dy) := F_P(dx, dy) - B_P(dx, dy) = \pi(dx) P(x,dy) - \pi(dy) P(y, dx).
\end{equation}
 Then $\Gamma$ is a signed measure on $S \times S$, satisfying
\begin{equation} \label{eq:vorticity_measure_condition_1} \Gamma(A \times B) = -\Gamma (B \times A), \quad A, B \in \mathcal S,\end{equation}
and
\begin{equation} \label{eq:vorticity_measure_condition_2} \Gamma(A, S) = 0, \quad A \in \mathcal S.\end{equation}
We will call a signed measure $\Gamma$ on $S \times S$ satisfying~\eqref{eq:vorticity_measure_condition_1},~\eqref{eq:vorticity_measure_condition_2} a \emph{vorticity measure}. If $\Gamma$ is related to $\pi$ and $P$ by~\eqref{eq:vorticity_measure}, it is called the \emph{vorticity measure of $(P,\pi)$}.

Let $Q(x, dy)$ and let $F_Q$ and $B_Q$ as defined above with $P$ replaced by $Q$. Let $\Gamma$ be a vorticity measure. The Markov chain $Q$ will play the role of proposal chain, and $\Gamma$ the role of target vorticity.

\begin{definition}[Absolute continuity of $\Gamma$]
We can use the Jordan decomposition \cite[Section 6.6]{Rudin1987} to decompose $\Gamma$ into two non-signed measures $\Gamma^+ := \half (|\Gamma| + \Gamma)$ and $\Gamma^- = \half (|\Gamma| - \Gamma)$, so that $\Gamma = \Gamma^+ - \Gamma^-$. We say that $\Gamma$ is \emph{absolutely continuous} with respect to some measure $M$ on $S \times S$, denoted by $\Gamma \ll M$, if $\Gamma^- \ll M$ and $\Gamma^+ \ll M$. If $\Gamma \ll M$, we define the Radon-Nikodym derivative of $\Gamma$ with respect to $M$ by
\[ \frac{d \Gamma}{d M}(x,y) = \frac{d \Gamma^+}{d M}(x,y) - \frac{d \Gamma^-}{d M}(x,y), \quad x, y \in S.\]
\end{definition}

Assuming $B_Q \ll F_Q$ and $\Gamma \ll F_Q$, we define the \emph{non-reversible Hastings ratio}
\begin{equation} \label{eq:general_hastings_ratio} R(x,y) := \frac{d \Gamma}{d F_Q}(x,y) + \frac{d B_Q}{d F_Q}(x,y).\end{equation}
In order for $R(x,y)$ to be nonnegative, we have to impose the condition 
\begin{equation} \label{eq:vorticity_nonnegative_condition} \frac{d\Gamma}{d F_Q}(x,y)  + \frac{d B_Q}{d F_Q}(x,y) \geq 0, \quad x, y \in S.\end{equation}

\begin{lemma}
\label{lem:Hastings_ratio_symmetric}
Suppose $Q$ and $\pi$ are such that $F_Q$ and $B_Q$ are equivalent measures on $S \times S$. Suppose that $\Gamma$ is a vorticity measure such that $\Gamma$ is absolutely continuous with respect to $F_Q$. Suppose~\eqref{eq:vorticity_nonnegative_condition} is satisfied. Then for all $x, y \in S$, $R(y,x) \geq 1$ if and only if $R(x,y) \leq 1$.
\end{lemma}

\begin{proof}
Write $\Gamma = \Gamma^+ - \Gamma^-$ for the Jordan decomposition of $\Gamma$, i.e. Since $\Gamma(A,B) = -\Gamma(B,A)$, it follows that $|\Gamma|(A,B) = |\Gamma|(B,A)$ and $\Gamma^+(A,B) = \Gamma^-(B,A)$ for $A, B \in \mathcal S$. Suppose for $x, y\in S$, 
$R(x,y) \leq 1$, so that
\begin{equation}\label{eq:Hastings_ratio_leq_1} 1 \geq \frac{ d (\Gamma + B_Q)}{d F_Q}(x,y) = \frac{d \Gamma^+}{d F_Q}(x,y) - \frac{d \Gamma^-}{d F_Q} (x,y) + \frac {d B_Q}{d F_Q}(x,y).\end{equation}
We compute
\begin{align*} R(y,x) & = \frac{d (\Gamma + B_Q)}{d F_Q}(y,x) = \frac{d \Gamma^+}{d F_Q}(y,x) - \frac{d \Gamma^-}{d F_Q}(y,x) + \frac{d B_Q}{d F_Q}(y,x) \\
 & = \frac{d \Gamma^-}{d B_Q}(x,y) - \frac{d \Gamma^+}{d B_Q}(x,y) + \frac{d F_Q}{d B_Q}(x,y) \\
 & = \frac{d F_Q}{d B_Q}(x,y) \left( \frac{d \Gamma^-}{d F_Q}(x,y) - \frac{d \Gamma^+}{d F_Q}(x,y) + 1 \right) \\
& \geq \frac{d F_Q}{d B_Q}(x,y) \frac {d B_Q}{d F_Q}(x,y) = 1,
\end{align*}
where~\eqref{eq:Hastings_ratio_leq_1} was used to establish the inequality.
\end{proof}

Define the \emph{acceptance probability} by
\begin{equation}
 \label{eq:general_acceptance} 
 A(x,y) = \min(1, R(x,y)), \quad x, y \in S, 
\end{equation}
and define a transition kernel $P$ by
\begin{equation}
\label{eq:general_probability_kernel}
 P(x, dy) = A(x,y) Q(x, dy) + \left( 1 - \int_S A(x,z) Q(x, dz) \right) \delta_x(dy).
\end{equation}

\begin{lemma}
\label{lem:P_has_vorticity_Gamma}
Suppose $Q$ and $\pi$ are such that $F_Q$ and $B_Q$ are equivalent measures on $S \times S$. Suppose that $\Gamma$ is a vorticity measure such that $\Gamma$ is absolutely continuous with respect to $F_Q$. Suppose~\eqref{eq:vorticity_nonnegative_condition} is satisfied. Let $P$ be defined by~\eqref{eq:general_acceptance},~\eqref{eq:general_probability_kernel}. The vorticity measure of $(P, \pi)$ is $\Gamma$.
\end{lemma}

\begin{proof}
We should check that for $A, B \in \mathcal S$,
\begin{equation} \label{eq:verify_vorticity} \Gamma(A, B) =  \int_A \left( \int_B P(x, dy) \right) \pi(dx) - \int_B \left( \int_A P(y, dx) \right) \pi (dy).\end{equation}

Note that, for some measurable function $h : S \rightarrow \R$,
\[ \int_A \left( \int_B  h(x) \delta_x(dy)\right) \pi (dx) = \int_A h(x) \1_{x \in B} \, \pi (dx) = \int_{A \cap B} h(x) \, \pi(dx).\]
This shows that the measure $\left( 1 - \int_S A(x,z) Q(x, dz) \right) \delta_x(dy) \pi(dx)$ in~\eqref{eq:general_probability_kernel} is symmetric with respect to interchanging $x$ and $y$, and therefore has no contribution in the verification of~\eqref{eq:verify_vorticity}.

Using Lemma~\ref{lem:Hastings_ratio_symmetric} there are no conceptual difficulties left in verifying that
\[  \int_A \left( \int_B P(x, dy) \right) \pi(dx) - \int_B \left( \int_A P(y, dx) \right) \pi (dy) = \Gamma(A,B),\]
in a similar way as in the proof of Lemma~\ref{lem:P_Gamma_has_vorticity_Gamma}.
\end{proof}

\begin{theorem}
\label{thm:general_case}
Suppose $Q$ and $\pi$ are such that $F_Q$ and $B_Q$ are equivalent measures on $S \times S$. Suppose that $\Gamma$ is a vorticity measure such that $\Gamma$ is absolutely continuous with respect to $F_Q$. Suppose~\eqref{eq:vorticity_nonnegative_condition} is satisfied. Let $P$ be defined by~\eqref{eq:general_acceptance},~\eqref{eq:general_probability_kernel}. Then $\pi$ is invariant for $P$ and $\Gamma$ is the vorticity measure of $(P,\pi)$.
\end{theorem}

\begin{proof}
Lemma~\ref{lem:P_has_vorticity_Gamma} establishes that $\Gamma$ is the vorticity measure of $(P,\pi)$. By~\ref{eq:vorticity_measure_condition_2}, for $A \in \mathcal S$,
\[ 0 = \Gamma(A,S) = \pi(A) - \int_S P(y, A) \pi(d y),\]
so that $\pi$ is invariant for $P$.
\end{proof}

\begin{proof}[Proof of Theorem~\ref{thm:regular_case}]
Define the signed measure $\Gamma$, and measures $F_Q$ and $B_Q$ on $S \times S$ by
\begin{align*} \Gamma(dx, dy) &= \gamma(x,y) \ dx \, dy, \\
 F_Q(dx, dy) & = \pi(dx) Q(x, dy), \quad B_Q(dx,dy) = \pi(dy) Q(y, dx).
\end{align*}
Then by~\eqref{eq:vorticity_kernel_condition_1},~\eqref{eq:vorticity_kernel_condition_2}, $\Gamma$ is a vorticity measure, and by assumptions~\eqref{eq:nonsingularity_condition_1} and~\eqref{eq:nonsingularity_condition_2}, $F_Q$ and $B_Q$ are equivalent measures, and $\Gamma$ is absolutely continuous with respect to $F_Q$. 
Furthermore $R(x,y)$ is a version of $\frac{d \Gamma}{d F_Q} + \frac{d B_Q}{d F_Q}$, and by assumption~\eqref{eq:nonnegativity_condition}, we have $\frac{d \Gamma}{d F_Q} + \frac{d B_Q}{d F_Q} \geq 0$. We see that all conditions of Theorem~\ref{thm:general_case} are satisfied, so that the stated results follow.
\end{proof}

\section{Proof of Theorem~\ref{thm:ornstein_uhlenbeck}}
\label{app:proof_ornstein_uhlenbeck}

Define an inner product $\langle \cdot, \cdot \rangle_V$ on $\R^n$ by $\langle x, y \rangle_V = \langle x, V^{-1} y\rangle$, where $\langle \cdot,\cdot \rangle$ is the Euclidean inner product. Let $\|\cdot\|_V$ denote the induced norm. Let $B :=  -(I+S)V^{-1}$.

\begin{lemma}
\label{lem:spectral-radius}
Suppose 
\begin{equation} \label{eq:stepsize} 0 < h < \frac 2 {\| V^{-1/2} (I-S^2) V^{-1/2} \|}.
\end{equation}
Then $\mathfrak{r}(I + h B) < 1$.
\end{lemma}

\begin{proof}
Suppose $\lambda \in \sigma(I + h B)$ and let $x \in \C^n$, $x \neq 0$, be an eigenvector corresponding to $\lambda$. Without loss of generality assume that $\|x\|_V = \|V^{-1/2} x\| = 1$.
We have
\[ \lambda = \langle (I+h B) x, x \rangle_V = \| x\|_V^2 - h \langle (I+S)V^{-1} x, V^{-1}x \rangle.\]
By skew-symmetry of $S$,
\begin{align*}  \Re \lambda & = \| V^{-1/2} x\|^2 - h \|V^{-1} x\|^2 = 1 - h \|V^{-1} x \|^2,
\end{align*}
and 
\begin{align*}
|\Im \lambda| & = h  | \langle S V^{-1} x, V^{-1} x \rangle |.
\end{align*}
Hence
\begin{align*} |\lambda|^2 & = \left( \| V^{-1/2} x \|^2 - h \|V^{-1} x \|^2 \right)^2 + h^2 \langle S V^{-1} x, V^{-1} x \rangle|^2, \\
 & = 1 - 2 h \|V^{-1} x \|^2 + h^2 \left( \|V^{-1} x\|^4 + \langle S V^{-1} x, V^{-1} x \rangle^2 \right).
\end{align*}
Hence the requirement $|\lambda| < 1$ translates into the inequality
\[ h^2 \left( \|V^{-1} x\|^4 + \langle S V^{-1} x, V^{-1} x \rangle^2 \right) < 2 h \|V^{-1} x \|^2,\]
or, after some rearranging,
\begin{equation} \label{eq:stepsize_intermediate} h < \frac{2}{\left( \|V^{-1} x\|^2 + \frac{\langle  S V^{-1} x, V^{-1} x \rangle^2}{\|V^{-1} x\|^2} \right) }.\end{equation}
Looking at the denominator, using Cauchy-Schwartz,
\begin{align*} \|V^{-1} x\|^2 + \frac{\langle  S V^{-1} x, V^{-1} x \rangle^2}{\|V^{-1} x\|^2}  
 & \leq \|V^{-1} x \|^2 + \|S V^{-1} x\|^2 = \langle V^{-1} x, V^{-1} x \rangle + \langle S V^{-1} x, S V^{-1} x \rangle \\
 & = \langle V^{-1} x, (I - S^2) V^{-1} x \rangle = \langle V^{-1/2} x, V^{-1/2} (I-S^2) V^{-1/2} V^{-1/2} x \rangle \\
& \leq \| V^{-1/2} (I-S^2) V^{-1/2} \| \|V^{-1/2} x \|^2 = \| V^{-1/2} (I-S^2) V^{-1/2} \|.
\end{align*}
It follows that if $h$ satisfies~\eqref{eq:stepsize}, then it satisfies~\eqref{eq:stepsize_intermediate}, and therefore $\mathfrak{r}(I + h B) < 1$.
\end{proof}

If~\eqref{eq:stepsize} holds, by \cite[Theorem 13.2.1]{LancasterTismenetsky1985} there exists a unique solution $R = R(\sigma)$ to the discrete Lyapunov equation~\eqref{eq:discrete_lyapunov}. Recall $f(x,y) = \rho(x) q(x,y)$, where $\rho$ is the density of $N(0, R)$ and $q(x,\cdot)$ the density function of $N((I+hB)x, (2h \sigma^2) I)$. Hence $f$ is a Gaussian density function with mean zero and covariance matrix
\begin{equation} \label{eq:M} M := \begin{pmatrix} R & R(I+h B') \\ (I+h B) R & 2 h \sigma^2 I + (I + h B) R(I+h B') \end{pmatrix} = \begin{pmatrix} R & R(I+h B') \\ (I+h B) R & R \end{pmatrix}.\end{equation}

\begin{lemma}
 \label{lem:det_M}
Suppose~\eqref{eq:stepsize} holds. Then 
$\det M = (2h \sigma^2)^n \det (R)$
\end{lemma}
\begin{proof}
By standard result on determinants of block matrices,
\begin{align*} \det M & = \det \begin{pmatrix} R & R(I+h B') \\ (I+h B)R & R \end{pmatrix} = \det (R) \det \left( R - (I+h B) R R^{-1} R(I+h B') \right) \\
& = \det (R) \det (R - (I+h B) R(I+h B')).\end{align*}
In the argument of the second determinant we recognize~\eqref{eq:discrete_lyapunov}, from which we obtain
\[ \det M = \det( R) \det(2h \sigma^2 I_n) = (2h \sigma^2)^n \det(R).\]
\end{proof}

Let $\preceq$ denote the partial ordering of positive definite matrices, i.e. $A \succeq B$ if $A - B$ is positive semidefinite.
\begin{lemma}
Suppose~\eqref{eq:stepsize} holds. Then 
\label{lem:lower_bound_R}
 $R \succeq \sigma^2 V.$
\end{lemma}
\begin{proof}
Expanding~\eqref{eq:discrete_lyapunov} gives
\[ R = 2 h \sigma^2 I + R + h BR + h R B' + h^2 B R B',\] or equivalently 
\[ 0 = 2 \sigma^2 I + (BR + R B') + h BR B' = 2 \sigma^2 I + (BR + R B') + T,\] where $T := h BRB' \succeq 0$.
It follows that $R$ satisfies the continuous time Lyapunov equation
\[ BR + RB' = -(2 \sigma^2 I + T)\]
with solution (see e.g. \cite[Theorem 13.1.1]{LancasterTismenetsky1985})
\[ R = \int_0^{\infty} e^{B s}(2 \sigma^2 I+T) e^{B's} \ ds \succeq \int_0^{\infty} e^{Bs}(2 \sigma^2 I) e^{B's} \ d s = \sigma^2 V,\]
where the last equality follows because $V$ satisfies the continuous time Lyapunov equation 
$B V + V B' = - 2 I$.
\end{proof}

\begin{lemma}
\label{lem:R_estimates}
\begin{itemize}
\item[(i)] For all $x \in \R^n$, $\langle x, B x\rangle_V \leq -  \|x\|^2_V / \| V\|$.
\item[(ii)] For all $s \geq 0$, $\| e^{Bs}\|_V \leq e^{-s / \| V\|}$.
\end{itemize}
Define constants $0 \leq C_1 < C_2$ by~\eqref{eq:constants}.
\begin{itemize}
\item[(iii)] For 
$h < \frac{2}{C_2}$ 
we have 
\[ \| B R  B'\| \leq \frac { 2 \sigma^2 C_1}{2 - h C_2}.\]
\item[(iv)] For $h < \frac{2}{C_2}$, we have
\[ R \preceq \sigma^2 \left( \frac{ 2 - h (C_2 - C_1)}{2 - h C_2} \right) V.\]
\end{itemize}

\end{lemma}
\begin{proof} 
Denote $\omega := \min \{ \lambda : \lambda \in \sigma(V^{-1})\} = 1/\|V\|$.
\begin{itemize}
\item[(i)]
Since $S$ is skew-symmetric,
\[ \langle x, B x\rangle_V = - \langle V^{-1} x, (I+S)V^{-1} x\rangle = -\langle V^{-1} x, V^{-1} x \rangle = - \| V^{-1/2} y\|^2,\]
where $y = V^{-1/2} x$.
Now $\| V^{-1/2} y\|\geq (1/\|V\|^{1/2}) \| y \|$, so we conclude
\[ \langle x, Bx \rangle_V \leq - \| y\|^2 /\|V\|  = - \| x\|_V^2/\|V\|.\]
 \item[(ii)] This follows immediately from (i) and the fact that a dissipative operator (here: $B + \frac 1 {\|V\|} I$) generates a contraction semigroup \cite[Section IX.8]{Yosida1980}.
 \item[(iii)]
 First note that 
\[ \| V^{-1/2} (I-S^2) V^{-1/2} \| = \| V^{-1/2} (I+S) V^{-1/2} V V^{-1/2} (I-S) V^{-1/2} \| \leq C_2.\] Therefore if $h < \frac 2 {C_2}$, then~\eqref{eq:stepsize} holds, and therefore $R$ is well-defined.
From the proof of Lemma~\ref{lem:lower_bound_R}, $R = \sigma^2 V + h \int_0^{\infty} e^{Bs} B R B' e^{B's} \ d s$. Hence, using (ii),
\begin{align*}
 \| B R B' \|_V & =  \left \| B \left(\sigma^2 V +  h \int_0^{\infty} e^{ Bs} B R B' e^{B's} \ d s \right) B' \right\|_V \\
 & \leq \sigma^2 \| B V B' \|_V + h \| B \|^2_V  \left( \int_0^{\infty} \| e^{Bs}\|^2_V \ d s\right)  \|BRB'\|_V \\
 & \leq  \sigma^2 \| B V B'\|_V + \half h  \| B \|^2_V \|V\|   \|BRB'\|_V.
\end{align*}
The result follows after rearranging, using the following equality (which holds for any matrix $K$) to express $\|B\|_V$ in terms of the $\|\cdot\|$-norm:
\[ \| K\|_V = \sup_{x \neq 0} \frac{ \| V^{-1/2} K x\|}{ \| V^{-1/2} x\|} = \sup_{y \neq 0} \frac{ \| V^{-1/2} K V^{1/2} y\|}{\|y\|}  = \| V^{-1/2} K V^{1/2} \|.\]
\item[(iv)]
By (iii), using that $\mathfrak{r}(BRB') \leq ||| B R B'|||$ for \emph{any} matrix-norm $|||\cdot|||$,
\[ BRB' \preceq \left(\frac { 2 \sigma^2 C_1}{2 - h C_2}\right) I.\]
 Hence
\begin{align*} R & = \sigma^2 V + h  \int_0^{\infty} e^{Bs}B RB'  e^{B's} d s \,  \preceq \sigma^2 V+ \left(\frac { 2  h \sigma^2  C_1}{2  - h C_2}\right) \int_0^{\infty} e^{ Bs} e^{  B' s} \ d s \\
& = \sigma^2 \left(1 + \frac{ h  C_1}{2  - h C_2} \right)V,\end{align*}
which is equivalent to the stated result.
\end{itemize}
\end{proof}

\begin{proof}[Proof of Theorem~\ref{thm:ornstein_uhlenbeck}]

The density $f$ is given by
\begin{equation} \label{eq:f} f(x,y) = (2\pi)^{-n} (\det M)^{-1/2} \exp \left( - \half \begin{pmatrix} x  \\ y \end{pmatrix}' M^{-1} \begin{pmatrix} x \\y \end{pmatrix} \right).\end{equation}
Using~\eqref{eq:discrete_lyapunov}, it can be verified that 
\begin{equation} \label{eq:M_inv} M^{-1} = \begin{pmatrix} R^{-1} + \frac 1 {2h \sigma^2} (I +h  B')(I+h  B) & -\frac 1 {2 h \sigma^2} ( I+h  B)'\\-\frac 1{2 h \sigma^2}(I+h  B) & \frac 1 {2 h \sigma^2} I \end{pmatrix}.\end{equation}
We have the following expressions for the target density $\pi$ and the transition density $q$ with respect to Lebesgue measure:
\begin{align*} \pi(x) & = \frac 1{ (2 \pi)^{n/2} \det(V)^{1/2}}  \exp(- \half \langle x, V^{-1} x \rangle), \quad \mbox{and} \\
q(x,y) & = \frac{1}{(2 \pi)^{n/2} (2 h \sigma^2)^{n/2}} \exp (-|y - (I+ h  B)x|^2/(4 h \sigma^2)). 
\end{align*}
Multiplication gives
\begin{align*}
 \pi(x) q(x,y) =\frac 1 { (2 \pi)^n \det(V)^{1/2} (2 h \sigma^2)^{n/2}}\exp\left( -\half \xi' N^{-1} \xi \right),
\end{align*}
where
\begin{equation} \label{eq:N_inv} \xi = \begin{pmatrix} x \\ y \end{pmatrix} \quad \mbox{and} \quad N^{-1} = \begin{pmatrix} V^{-1} + \frac 1 {2h \sigma^2} (I+h  B')(I+h  B) & -\frac 1 {2h \sigma^2 } (I+ h   B') \\
             -\frac 1{2h \sigma^2} (I+h  B) & \frac 1 {2h \sigma^2} I \end{pmatrix}.\end{equation}
To satisfy~\eqref{eq:nonnegativity_condition}, we require that $c \gamma(y,x) + \pi(x) q(x,y) \geq 0$ for all $x, y \in \R^n$ and some constant $c > 0$. 
We compute
\begin{align*}
 & \pi(x) q(x,y)  + c \gamma(y,x)  =  \pi(x) q(x,y)  + c (f(y,x) - f(x,y) ) \geq  \pi(x) q(x,y) - c f(x,y) \\
 & = (2 \pi)^{-n} (\det V)^{-1/2} (2h \sigma^2)^{-n/2} \exp \left( -\half \xi' N^{-1} \xi \right) \\
 & \quad \quad - c (2 \pi)^{-n} \det (M)^{-1/2} \exp \left( - \half \xi' M^{-1} \xi \right).
\end{align*}
By Lemmas~\ref{lem:det_M} and~\ref{lem:lower_bound_R}, we have 
\[ \det (M)^{-1/2} = \frac{1}{(2h \sigma^2)^{n/2}} \det(R)^{-1/2} \leq \left( \frac{1}{(2h)^{n/2} (\sigma^2)^n (\det V)^{1/2}} \right),\] so that, for $c \leq \sigma^n$,
\begin{align*} c \gamma(y,x) + \pi(x) q(x,y) &  \geq  k \left[ \exp \left(-\half \xi' N^{-1} \xi \right) - \exp \left(-\half  \xi' M^{-1} \xi \right) \right] \\
& =k  \exp \left( - \half \xi' M^{-1} \xi \right) \left[ \exp \left( \half \xi' M^{-1} \xi -\half \xi' N^{-1} \xi  \right) - 1 \right], \end{align*} 
where $k = \left( (2 \pi)^{2n} (2h \sigma^2)^n \det V \right)^{-1/2}$.
The last factor is nonnegative for all $x,y$ if and only if $M^{-1} - N^{-1}$ is positive semidefinite. By~\eqref{eq:M_inv} and~\eqref{eq:N_inv}, we have
\[ M^{-1} - N^{-1} = \begin{pmatrix} R^{-1} -V^{-1} & 0 \\ 0 & 0 \end{pmatrix}.\]
Using Lemma~\ref{lem:R_estimates} (iv), we find that for the specified values of $h$ and $\sigma$, $R \preceq V$ and therefore (by \cite[Corollary 7.7.4]{HornJohnson1990}), $R^{-1} - V^{-1} \succeq 0$.
\end{proof}


\newcommand{\etalchar}[1]{$^{#1}$}

\end{document}